\documentclass[12pt,letterpaper]{amsart}
\usepackage{amsmath,amssymb,amsfonts, float}
\usepackage[all]{xy}
\usepackage{caption}
\usepackage{enumerate}
\usepackage{hyperref} 
\usepackage{bm}
\usepackage{graphicx}

\usepackage{xcolor}

\usepackage{hyperref}

\usepackage[normalem]{ulem}

\newtheorem{thm}{Theorem} 
\newtheorem*{thm*}{Theorem} 
\newtheorem{prop}[thm]{Proposition}
\newtheorem{lem}[thm]{Lemma}
\newtheorem{cor}[thm]{Corollary}

\theoremstyle{definition}
\newtheorem{definition}[thm]{Definition}
\newtheorem{expl}[thm]{Example}

\newtheorem{rem}[thm]{Remark}

\DeclareMathOperator{\R}{\mathbb{R}}
\DeclareMathOperator{\C}{\mathbb{C}}

\DeclareMathOperator{\N}{\mathbb{N}}

\DeclareMathOperator{\Spec}{\text{Spec}}

\DeclareMathOperator{\conc}{\mathbin{\ast}}

\DeclareMathOperator{\Tr}{\text{Tr}}
\DeclareMathOperator{\RG}{\text{RG}}

    \DeclareFontFamily{U}{wncy}{}
    \DeclareFontShape{U}{wncy}{m}{n}{<->wncyr10}{}
    \DeclareSymbolFont{mcy}{U}{wncy}{m}{n}
    \DeclareMathSymbol{\Sha}{\mathord}{mcy}{"58}

\numberwithin{equation}{section}

\renewcommand{\i}{\mathrm{i}}

\DeclareSymbolFont{bbold}{U}{bbold}{m}{n}
\DeclareSymbolFontAlphabet{\mathbbold}{bbold}
\newcommand{\ind}{\bm{1}}

\title{Joins of circulant matrices}
  \author{ Jacqueline \DJ o\`{a}n, J\'an Min\'a\v{c}, Lyle Muller,
  \\ Tung T. Nguyen, Federico W. Pasini}

\date{\today}

\begin{document}
\maketitle
\begin{abstract}
We study the spectrum of the join of several circulant matrices. We apply our results to compute explicitly the spectrum of certain graphs obtained by joining several circulant graphs.
\end{abstract}

\section{Introduction} 
Circulant matrices provide a nontrivial, elegant, and simple set of objects in matrix theory. They appear quite naturally in many problems in spectral graph theory (see \cite{[G1]}, \cite{[G2]}, \cite{[G3]}, \cite{[G4]}, \cite{[G5]}, \cite{[G6]}) and non-linear dynamics (see  \cite{[K]}, \cite{[LMN]}, \cite{[Townsend]}). The Circulant Diagonalization Theorem describes the eigenspectrum and eigenspaces of a circulant matrix explicitly via the discrete Fourier transform. Consequently, many problems involving circulant matrices have closed-form or analytical solutions.

For example, in many applications, a natural model of a network is a \textit{ring graph}, in which nodes are regularly placed along a circle and, for a fixed number $m$, each node is connected to its $m$ closest neighbours on each side. Networks such as this can be represented by adjacency matrices which are circulant, which opens the possibility for exact solutions for problems involving the structure or dynamics of these networks. More generally, a graph which has a circulant adjacency matrix with respect to a suitable ordering of the vertices is called a \textit{circulant graph.}

Many real-world networks, however, display structure beyond that of circulant networks. For example, networks may be composed of several smaller modules, joined together in some way (see the final section for a particular example). From both a theoretical and an applied perspective, it is interesting and important to study the spectra of graphs obtained by joining together smaller subgraphs.

The combination of these previous observations naturally led us to investigate the spectrum of networks composed of several circulant graphs. While in general it is impossible to relate the spectrum of a graph with the spectra of its subgraphs, joins of circulant graphs provide an exception. Here we present a study of these spectra, and some applications. These results can provide analytical insight into the dynamics of composite networks (see e.g.\cite{[KO4]}), which will be the subject of future work.

More precisely, we generalize the Circulant Diagonalization Theorem to the \textit{joins of several circulant matrices}, by which we mean matrices of the shape
\begin{equation}
    \label{eq:d-join}\tag{$\ast$}
    A=\left(\begin{array}{c|c|c|c}
C_1 & a_{1,2}\ind & \cdots & a_{1,d}\ind \\
\hline
a_{2,1}\ind & C_2 & \cdots & a_{2,d}\ind \\
\hline
\vdots & \vdots & \ddots & \vdots \\
\hline
a_{d,1}\ind & a_{d,2}\ind & \cdots & C_d
\end{array}\right),
\end{equation}
where, for each $1 \leq i,j \leq d$, $C_i$ is a circulant matrix of size $k_i \times k_i$ (with complex entries), and $a_{i,j}\ind$ is a $k_i \times k_j$ matrix with all entries equal to a constant $a_{i,j}\in\mathbb{C}$. We remark that, to simplify notation, $\ind$ is used as the common symbol for all matrices with all entries equal to $1$, independently of their sizes. However no confusion should occur as the submatrices $a_{i,j} \ind$ are uniquely determined. 

Our main theorem is
\begin{thm*}\label{thm:main thm}
The spectrum of a matrix $A$ as in \eqref{eq:d-join} is the union of the following multisets 
\[ \Spec\left(\overline{A}\right) \cup \bigcup_{i=1}^{d} \left\{ \lambda^{C_i}_{j}| 1 \leq j \leq k_i-1 \right\}, \] 
where $\overline{A}$ is an explicit $d \times d$ matrix, whose entries are the row sums of the blocks of $A$, and the $\lambda^{C_i}_{j}$'s are the eigenvalues of each circulant block $C_i$, except for the eigenvalue given by the row sum. 
Furthermore, a generalized eigenbasis of $A$ can be directly obtained from eigenbases of the circulant blocks and a generalized eigenbasis of $\overline{A}$. In particular, $A$ is diagonalizable if and only if $\overline{A}$ is.
\end{thm*}

This theorem completely solves our main problem of characterization of spectrum of the join of $d$ circulant matrices. We note that the methods in this article can be generalized to a wider class of matrices, namely normal matrices with constant row sums. This extension will be discussed in a separate paper in preparation.

The structure of this article is as follows. In Section $2$, we illustrate the join of two circulant matrices. This serves as a motivation for our study as well as to guide the readers to the more general case. In Section $3$, we give the complete proof of the main theorem, which consists of several steps. First, we show how to extend eigenvectors of a circulant block to eigenvectors of the join. Secondly, we show that the generalized eigenspaces of $\overline{A}$ lift to the generalized eigenspaces of $A$. Finally, we prove that the collection of (generalized) eigenvectors for $A$, obtained from the previous two processes, form a generalized eigenbasis. In Section $4$, we discuss some applications of our results to spectral graph theory. In the final section, we use the main theorem to study the dynamics of networks of coupled oscillators. Specifically, we construct a family of networks of Kuramoto oscillators with non-trivial equilibrium points.

\section{Motivation: the join of two circulant matrices}\label{sec:2 circulant}

A special instance of \textit{joining circulant matrices} arises when we study the removal of one (directed) cycle from a complete graph. Recall that the \textit{complete graph} of size $n$, denoted $K_n$, is the simple graph with an edge between any two distinct nodes. Its adjacency matrix $A$ is given by 
\[ A_{ij} = \begin{cases} 0 & \text{ if } i = j \\ 1 & \text{otherwise.} \end{cases}\]
Moreover, a (directed) cycle of length $k$, or $k$-cycle, denoted $C_k$, is the simple graph on $k$ nodes, in which the nodes can be ordered in such a way that each node is connected only with the subsequent one, and the last one only with the first one. Its adjacency matrix $A$ is given by
\[ A_{ij} = \begin{cases} 1 & \text{ if } j=i+1 \text{ or } (i,j)=(k,1) \\ 0 & \text{otherwise.} \end{cases}\]
Finally, the complement of a graph $G$ is the graph $G^c$ with the same vertices as $G$ and which has the edge between two distinct vertices if and only if that edge is not in $G$. In other words, the adjacency matrix $A^c$ of $G^c$ is related to the adjacency matrix $A$ of $G$ by $A^c=\mathbf{1}-I-A$, where $\mathbf{1}$ is a square matrix of ones and $I$ is an identity matrix, of suitable size.
In particular, the adjacency matrices of complete graphs, cycles, and complements of cycles are all circulant.

We illustrate the general phenomenon of cycle removal on a small concrete example. Let us remove a $3$-cycle $C_3$ from the complete graph $K_{8}$ with $8$ nodes, and call the resulting graph $\mathcal{K} = K_8-C_3$. We choose to remove the cycle \[(1,2),(2,3),(3,1)\] which is highlighted in red in the figure below.
\begin{figure}[H]
\centering
\includegraphics[scale=0.8]{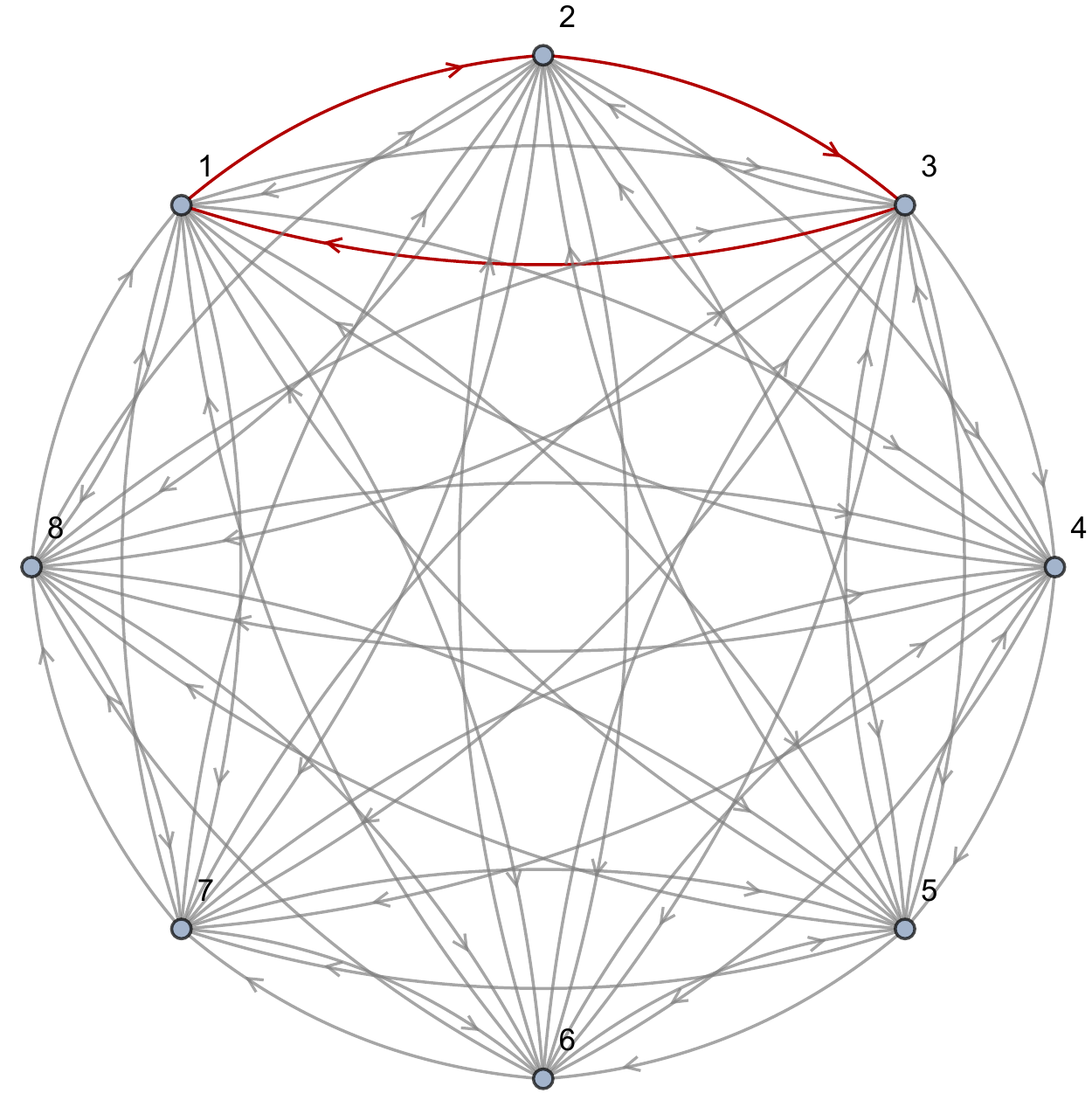}
\caption{The complete graph $K_8$ with a cycle of length $3$ removed.}
\label{fig:complete}
\end{figure}

We note that removing any other cycle of length 3 would produce an isomorphic graph. Another representation of this graph is depicted in the figure below. We have two circulant graphs $G$ and $H$ (in green and grey respectively)
and all nodes from each ring graph are adjacent to all nodes of the other ring graph. This is an instance of the \textit{join of two circulant graphs}, which we will define in Section 4.
\begin{figure}[H]
\centering
\includegraphics[width=0.5\linewidth]{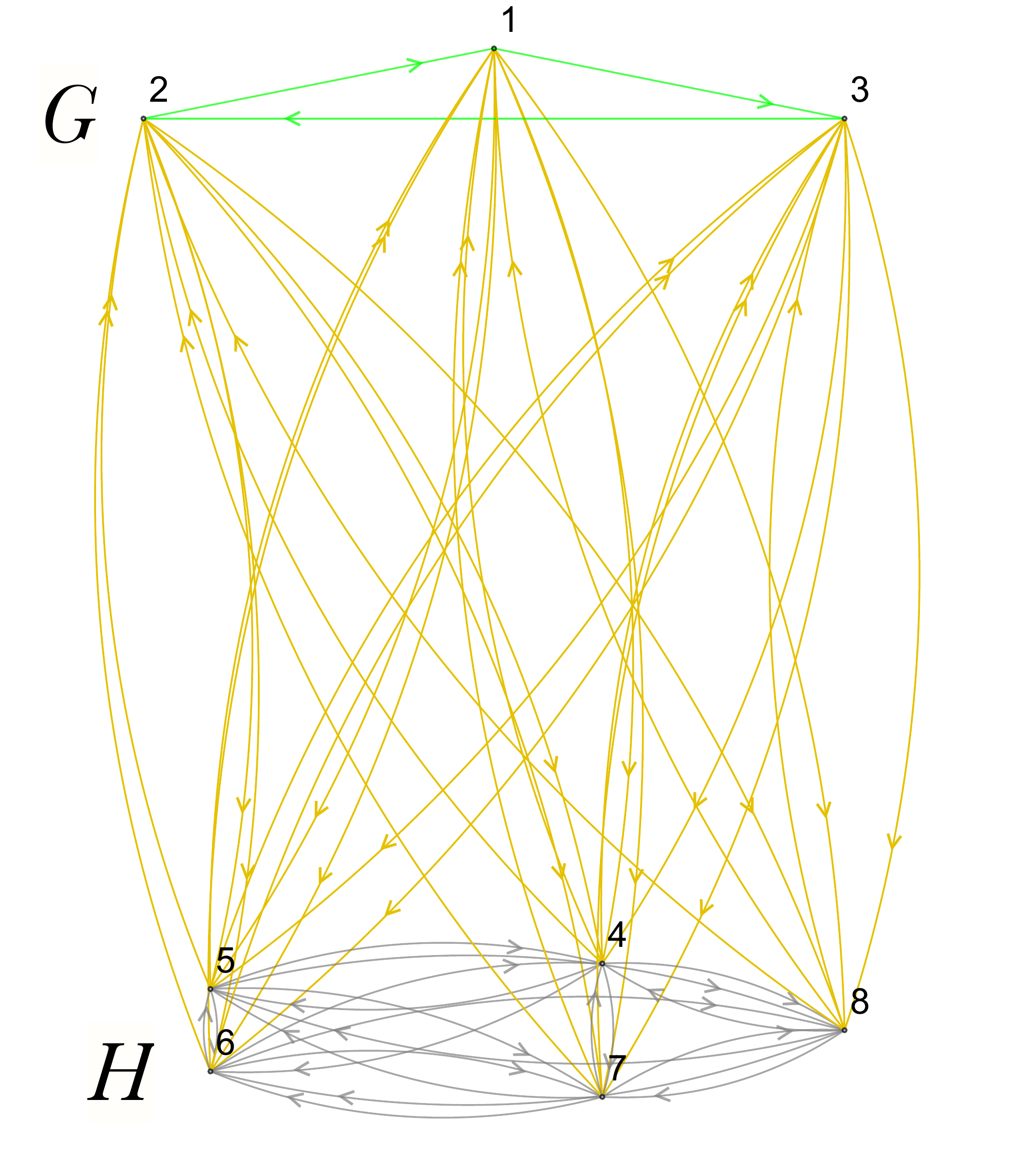}
\caption{The join of two circulant graphs $G$ and $H$.}
\label{fig:join}
\end{figure}
In matrix terms, the adjacency matrix of $\mathcal{K}$ is a block matrix, with circulant diagonal blocks and $1$ everywhere else.
\[
\left( \begin{array}{ccc | ccccc }
 0 & 0 & 1 & 1 & 1 & 1 & 1 & 1 \\
 1 & 0 & 0 & 1 & 1 & 1 & 1 & 1 \\
 0 & 1 & 0 & 1 & 1 & 1 & 1 & 1 \\
 \hline
 1 & 1 & 1 & 0 & 1 & 1 & 1 & 1 \\
 1 & 1 & 1 & 1 & 0 & 1 & 1 & 1 \\
 1 & 1 & 1 & 1 & 1 & 0 & 1 & 1 \\
 1 & 1 & 1 & 1 & 1 & 1 & 0 & 1 \\
 1 & 1 & 1 & 1 & 1 & 1 & 1 & 0 \\
\end{array} \right)
\]
The position of the eigenvalues of the adjacency matrix of $\mathcal{K}$ in the complex plane highlights a nontrivial interplay between the eigenvalues of the two circulant blocks, which motivates our investigations of the eigenspectra of joins of circulant matrices.
%
\begin{figure}[H]
\centering
\includegraphics[width=\linewidth]{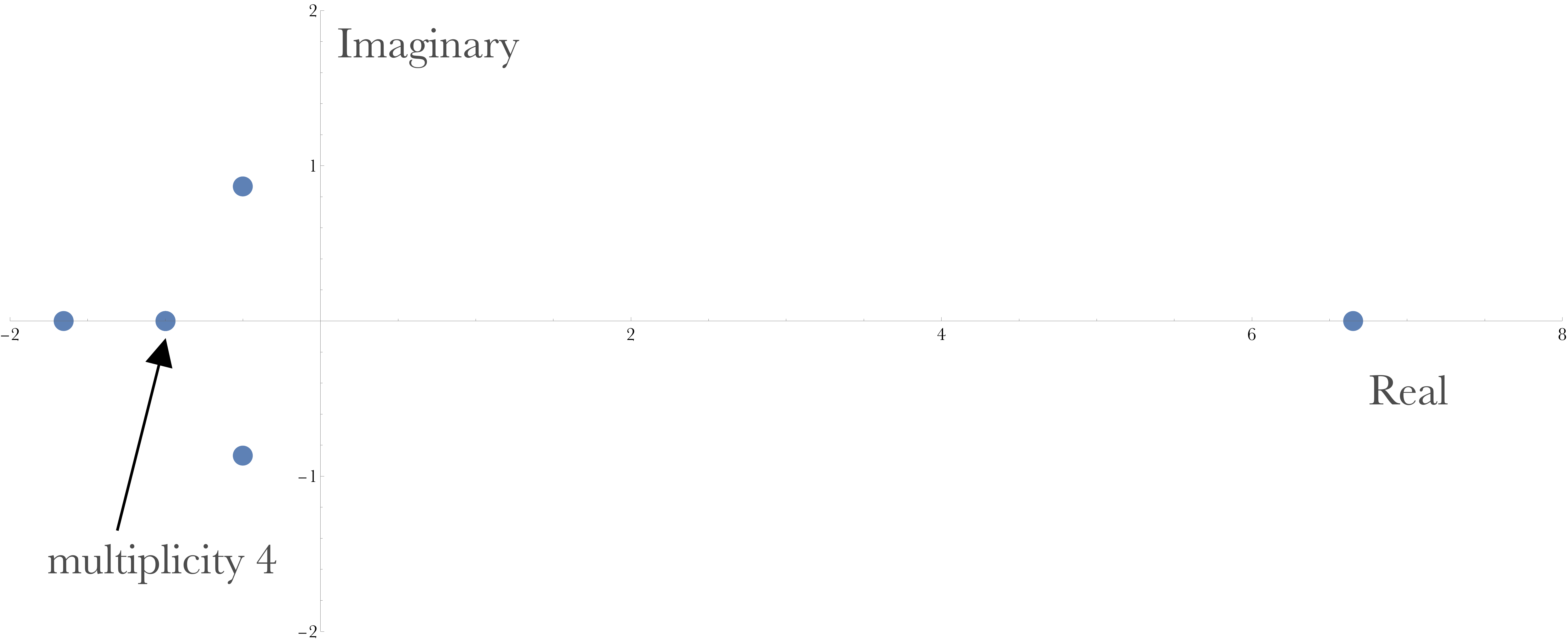}
\caption{The eigenvalues of the resulting graph.}
\label{fig:eigen}
\end{figure}

It is worth noticing that the eigenvalues for the graph $\mathcal{K}$ have been obtained through the software \texttt{Mathematica}, but in the course of the paper we will derive analytical expressions for them.

To begin our investigation, we recall the Circulant Diagonalization Theorem (see \cite{[Davis]} for a more thorough discussion about circulant matrices). 
In the following, $\omega_k$ denotes a fixed primitive $k$-th root of unity.
\begin{thm}[{Circulant Diagonalization Theorem, \cite{[Davis]}}] \label{prop: CDT}
Let 
\[ C={\begin{pmatrix}c_{0}&c_{k-1}&\cdots &c_{2}&c_{1} \\c_{1}&c_{0}&c_{k-1} &&c_{2}\\\vdots &c_{1}&c_{0}&\ddots &\vdots \\c_{k-2}&&\ddots &\ddots &c_{k-1}\\c_{k-1}&c_{k-2}&\cdots &c_{1}&c_{0}\\\end{pmatrix}}=Circ(c_0,c_1,\dots,c_{k-1}) \]
be the circulant matrix formed by the vector $(c_0,c_1,\dots,c_{k-1})^T\in\mathbb{C}^k$. Let 
\[ {\displaystyle v_{k,j}=\left(1,\omega_k ^{j},\omega_k ^{2j},\ldots ,\omega_k ^{(k-1)j}\right)^T,\quad j=0,1,\ldots ,k-1} .\] 
Then $v_{k,j}$ is an eigenvector of $C$ associated with the eigenvalue 
\[ \lambda _{j}=c_{0}+c_{k-1}\omega_k ^{j}+c_{k-2}\omega_k ^{2j}+\dots +c_{1}\omega_k ^{(k-1)j} \] 

\end{thm} 
\begin{rem}\label{rem:Fourier modes}
For any choice of $k\in\N\setminus\{0\}$, the vectors $v_{k,0},\dots,v_{k,k-1}$ are linearly independent. This can be seen by noticing that the matrix formed by the vectors is a Vandermonde matrix.
\end{rem}

In the following, the operator $\conc$ denotes vector concatenation:
\[
(x_1,\dots,x_m)^T\conc (y_1,\dots,y_n)^T = (x_1,\dots,x_m, y_1,\dots,y_n)^T
\]
\begin{prop}
Let $C$ be a $k \times k$ circulant matrix, $D$ be any $(n-k) \times (n-k)$ matrix, let $\bm{1}_{k_1,k_2}$ denote the $k_1\times k_2$ matrix entirely made of ones, and let $A$ be the $n \times n$ matrix
\[ A= \begin{pmatrix}
C & \bm{1}_{k, n-k} \\ 
\bm{1}_{n-k, k} & D
\end{pmatrix}. \]
For $1 \leq j \leq k-1$ let 
\[ w_j= (1, \omega_k^j, \omega_{k}^{2j}, \ldots, \omega_{k}^{(k-1)j}, \underbrace{0, \ldots, 0}_{n-k\text{ zeros}} )^T=v_{k,j}\conc \vec{0}_{n-k}.\] 
 Then $w_j$ is an eigenvector of $A$ associated with the eigenvalue 
\[ \lambda _{j}^C=c_{0}+c_{k-1}\omega _k^{j}+c_{k-2}\omega_k ^{2j}+\dots +c_{1}\omega_k ^{(k-1)j} \] 
\end{prop} 

\begin{proof}
When we directly calculate $Aw_j$ we see that the first $k$ elements of this vector are $C v_{k,j}$ and the remaining $n-k$ elements are equal to the sum 
\[t_j =\sum_{i=0}^{k-1} \omega_{k}^{i j}. \]

In other words, we have 
\[ Aw_j= Cv_{k,j} \conc  \underbrace{(t_{j} , t_{j}, \ldots, t_{j})^T}_{\text{$n-k$ terms}}.=\lambda^C_j v_j   \conc \underbrace{(t_{j} , t_{j}, \ldots, t_{j})^T}_{\text{$n-k$ terms}}\] 
Since, for $1 \leq j \leq k-1$,
\[ t_j =\sum_{i=0}^{k-1} \omega_{k}^{i j}=\frac{(\omega_{k}^j)^k-1}{\omega_k^j-1}=0,\]
it follows that
\( Aw_j=\lambda_j \omega_j.\) We conclude that $w_j$, $1 \leq j \leq k-1$, are eigenvectors of $A$ with associated eigenvalue $\lambda_{j}^C$ as asserted. 
\end{proof}

If $D$ is also circulant, $D= Circ(d_0, d_1, \ldots, d_{k_2-1})$ with $k_2=n-k$, an analogous argument applies. 
In summary, recalling Remark \ref{rem:Fourier modes} for the claim on linear independence, we have proved the following statement. 
\begin{prop} \label{prop:block_circulant}
Let $A$ be a $(k_1+k_2) \times (k_1+k_2)$ matrix of the form 
\[ A= \begin{pmatrix}
C & \bm{1}_{k_1, k_2} \\ 
\bm{1}_{k_2, k_1} & D
\end{pmatrix}, \]
with $C=Circ(c_0,\dots,c_{k_1-1})$ and $D=Circ(d_0,\dots,d_{k_2-1})$ circulant matrices of dimension $k_1 \times k_1$ and $k_2 \times k_2$ respectively. For $1 \leq j \leq k_1-1$ let 
\[ w_j= (1, \omega_{k_1}^j, \omega_{k_1}^{2j}, \ldots, \omega_{k_1}^{(k_1-1)j}, \underbrace{0 \ldots, 0} _{\text{$k_2$ zeros}})^T=v_{k_1,j}\conc \vec{0}_{k_2}  .\] 

 Then $w_j$ is an eigenvector of $A$ associated with the eigenvalue 
\[ \lambda _{j}^C=c_{0}+c_{k_1-1}\omega _{k_1}^{j}+c_{k_1-2}\omega_{k_1} ^{2j}+\dots +c_{1}\omega_{k_1} ^{(k_1-1)j}. \] 
For $1 \leq j \leq k_2-1$, let 
\[ z_j= (\underbrace{0, \ldots, 0}_{\text{$k_1$ zeros}}  ,1, \omega_{k_2}^j, \omega_{k_2}^{2j}, \ldots, \omega_{k_2}^{(k_2-1)j})^T = \vec{0}_{k_1} \conc v_{k_2,j}  .\] 
Then $z_j$ is an eigenvector associated with the eigenvalue 
\[ \lambda _{j}^D=d_{0}+d_{k_2-1}\omega _{k_2}^{j}+d_{k_2-2}\omega_{k_2} ^{2j}+\dots +d_{1}\omega_{k_2} ^{(k_2-1)j} \] 

Furthermore, the system of $k_1+k_2-2$ eigenvectors $\{w_{j} \}_{j=1}^{k_1-1} \cup \{z_j \}_{j=1}^{k_2-1}$ is linearly independent. 

\end{prop} 

In order to find the two remaining eigenvalues and corresponding eigenvectors of the matrix $A$, we introduce an auxiliary matrix. 
\begin{prop}
Keeping the notation of the previous proposition, let $C_{s}=\sum_{i=0}^{k_i-1} c_i$ be the sum of each row in $C$, and similarly let $D_{s}=\sum_{i=0}^{k_2-1} d_i$. Let us consider the $2 \times 2$ matrix 
\[ \overline A=\begin{pmatrix}
   C_{s} & k_2 \\
   k_1  &  D_s 
\end{pmatrix} .\] 

Let $(x,y) \in \C^2$ be an eigenvector for $\overline A$ with respect to an eigenvalue $\lambda$. Then 
\[v= (\underbrace{x, x, \ldots, x}_{\text{$k_1$ terms}},  \underbrace{y, y, \ldots, y}_{\text{$k_2$ terms}})^T\] 
is an eigenvector of $A$ with respect to the eigenvalue $\lambda$.

\end{prop}

\begin{proof}
We have 
\[ Av = (\underbrace{C_s x+k_2y, \ldots, C_s x+k_2y}_{\text{$k_1$ terms}}, \underbrace{k_1x+D_s y, \ldots, k_1 x+D_s y}_{\text{$k_2$ terms}})^T .\] 

By assumption, 
\( C_s x+k_2y=\lambda x, \) 
and 
\( D_s x +k_1 y = \lambda y .\) 

Therefore, we see that 
\[ Av=\lambda (\underbrace{x, x, \ldots, x}_{\text{$k_1$ terms}},  \underbrace{y, y, \ldots, y}_{\text{$k_2$ terms}})^T= \lambda v .\] 
\end{proof}

\begin{prop} \label{prop:diag}
Keeping the notation of the previous proposition, suppose further that $\overline A$ is diagonalizable with eigenvectors $(x_1,y_1)$ and $(x_2, y_2)$. Let 
\[v_1= (\underbrace{x_1, x_1, \ldots, x_1}_{\text{$k_1$ terms}},  \underbrace{y_1, y_1, \ldots, y_1}_{\text{$k_2$ terms}})^T,\qquad v_2= (\underbrace{x_2, x_2, \ldots, x_2}_{\text{$k_1$ terms}},  \underbrace{y_2, y_2, \ldots, y_2}_{\text{$k_2$ terms}})^T.\]

Then the system $\{w_{j} \}_{j=1}^{k_1-1} \cup \{z_j \}_{j=1}^{k_2-1} \cup \{v_1, v_2 \}$ of eigenvectors of $A$ is linearly independent. In other words, $A$ is diagonalizable by these eigenvectors.
\end{prop} 

\begin{proof}
For each $k$ let 
\[ E_k = \begin{pmatrix}1&1&1&\dots &1\\1& \omega_k&\omega_k^{2}&\dots &\omega_k^{k-1}\\1&\omega_k^2&\omega_k^{4}&\dots &\omega_k^{2(k-1)}\\\vdots &\vdots &\vdots &\ddots &\vdots \\1& \omega_k^{k-1}&\omega_k^{2(k-1)}&\dots &\omega_k^{(k-1)(k-1)}\end{pmatrix} \] 
be the matrix that is used to diagonalize a $k \times k$ circulant matrix, and $\widehat{E_k}$ be the submatrix of $E_k$ with the first column removed. Let $e_k=\det(E_k) \neq 0$.  The system $\{w_{j} \}_{j=1}^{k_1-1} \cup \{z_j \}_{j=1}^{k_2-1} \cup \{v_1, v_2 \}$ can be arranged to create the following matrix

\[
E=\left(\begin{array}{c|ccc|ccc|c}
\begin{array}{c}x_1\\\vdots\\x_1\end{array} && \widehat{E_{k_1}} &&& \begin{array}{c}x_2\\\vdots\\x_2\end{array} &&0 \\
\hline
\begin{array}{c}y_1\\\vdots\\y_1\end{array}  && 0 &&&\begin{array}{c}y_2\\\vdots\\y_2\end{array} &&\widehat{E_{k_2}} 
\end{array}\right)
\]
Using the Laplace expansion of the determinant (see \cite[Theorem 2.4.1]{[Prasolov]}),
we obtain the term $x_1y_2 e_{k_1}e_{k_2}$ as the product of the determinant of the left top corner block matrix of size $k_1 \times k_1$ with the determinant of the right down corner matrix of the size $k_2 \times k_2$. The only other non-zero summand in the Laplace expansion is the product
\[
\det \left(\left[\begin{array}{ccc|c}
\widehat{E_{k_1}} &&& \begin{array}{c}x_2\\\vdots\\x_2\end{array}  \end{array} \right] \right)\cdot
\det \left(\left[\begin{array}{c|ccc}
\begin{array}{c}y_1\\\vdots\\y_1\end{array} &&&
\widehat{E_{k_2}}   \end{array} \right] \right)=-x_2 y_1 e_{k_1} e_{k_2}. \]
Consequently,
\[ \det(E)=e_{k_1} e_{k_2} (x_1 y_2-x_2y_1)=e_{k_1} e_{k_2} \det \begin{pmatrix} x_1 & x_2 \\ y_1 & y_2 \end{pmatrix} \neq 0.\] 

\end{proof}

In addition, there is a relationship between the eigenvalues of $A$ and $\overline A$, to prove which we need a preliminary lemma.

\begin{lem} \label{lem:trace}
Let $M=Circ(m_0, \ldots, m_{k-1})$ be a circulant matrix. Let $M_s=\sum_{i=0}^{k-1} m_i$.  Let $\{\lambda_{j}^M \}_{j=0}^{k-1} $ be the set of eigenvalues of $M$ described in the Circulant Diagonalization Theorem \ref{prop: CDT}. Then 
\begin{enumerate}
\item $\sum_{j=1}^{k_1-1} \lambda_j^M=\Tr(M)-M_s.$
\item $\sum_{j=1}^{k_1-1} (\lambda_{j}^M)^2=\Tr(M^2)-M_s^2.$
\end{enumerate}
\end{lem}

\begin{proof}
Both equalities are direct consequences of the facts that, when $j=0$, $\lambda_{j}^M=M_s$, and that for all $k \geq 0$ 
\[ \Tr(M^k)=\sum_{j=0}^{k-1} (\lambda_{j}^M)^k .\] 
\end{proof}

\begin{prop} \label{prop:converse}
Keeping the notation of the previous proposition, let $\lambda_1, \lambda_2$ be the two remaining eigenvalues of $A$, that is, the eigenvalues not coming from the circulant blocks $C$ and $D$. Then $\lambda_1$ and $\lambda_2$ are eigenvalues of $\overline A$.
\end{prop}
\begin{proof}
It is enough to show that 
\[ \lambda_1+\lambda_2=C_s +D_s, \qquad\text{and}\qquad \lambda_1 \lambda_2=C_sD_s-k_1k_2 .\] 
First, by Proposition \ref{prop:block_circulant} we have 
\[ \lambda_1+\lambda_2+ \sum_{j=1}^{k_1-1} \lambda_{j}^C+\sum_{j=1}^{k_2-1} \lambda_{j}^D = \Tr(A)=\Tr(C)+\Tr(D) .\]
By Lemma \ref{lem:trace}, we have 
\[ \sum_{j=1}^{k_1-1} \lambda_{j}^C=\Tr(C)-C_s ,\qquad\text{and}\qquad \sum_{j=1}^{k_2-1} \lambda_{j}^D=\Tr(D)-D_s .\] 
Combining these equalities, we conclude that 
\[ \lambda_1+\lambda_2= C_s+D_s .\] 

To prove the equality $\lambda_1 \lambda_2=C_sD_s-k_1k_2$, we first compute $\lambda_1^2+\lambda_2^2$, using $A^2$. We have 

\[ A^2=\begin{pmatrix}
   C^2+k_2 \bm{1}_{k_1} & * \\
   * &  D ^2+k_1 \bm{1}_{k_2}
\end{pmatrix},\] 
where $\bm{1}_{k}$ denotes a $k\times k$ matrix with all entries equal to $1$.
This implies that 
\[ \Tr(A^2)=\Tr(C^2)+\Tr(D^2)+2k_1k_2. \] 
Additionally, we have 
\[ \Tr(A^2)=\lambda_1^2+\lambda_2^2 +\sum_{j=1}^{k_1-1} (\lambda_{j}^C)^2+\sum_{j=1}^{k_2-1} (\lambda_j^D)^2 ,\] 
\[ \Tr(C^2)= \sum_{j=1}^{k_1-1} (\lambda_{j}^C)^2+C_s^2, \] 

\[ \Tr(D^2)= \sum_{j=1}^{k_2-1} (\lambda_{j}^D)^2+D_s^2. \] 
Combining these equalities, we get 
\[ \lambda_1^2+\lambda_2^2=C_s^2+D_s^2+2k_1k_2 .\] 
Therefore, by Newton's formula we have 
\begin{align*}
\lambda_1 \lambda_2 &=\frac{1}{2} \left[(\lambda_1+\lambda_2)^2-\lambda_1^2-\lambda_2^2 \right] \\ &=\frac{1}{2} \left[(C_s+D_s)^2-(C_s^2+D_s^2+2k_1k_2)\right]\\ &=C_sD_s-k_1k_2.
\end{align*}
This completes the proof.
\end{proof}

We discuss a significant case in which $\overline A$ is diagonalizable. 
\begin{prop}\label{prop:explicit eigenvalues}
Keeping the notation of the previous proposition, suppose that $C_s$ and $D_s$ are real numbers (or complex numbers with the same real part, or with the same imaginary part). Then $\overline A$ is diagonalizable. Consequently, $A$ is diagonalizable by the system of eigenvectors discussed in proposition \ref{prop:diag}.
\end{prop} 

\begin{proof}

The characteristic polynomial of $\overline A$ is 
\[ X^2-(C_s+D_s)X+(C_sD_s-k_1k_2) .\] 

The discriminant of this polynomial is 
\[ \Delta(\overline A)=(C_s+D_s)^2-4(C_sD_s-k_1k_2)=(C_s-D_s)^2+4k_1 k_2 .\] 
Since $C_s-D_s$ is either real or purely imaginary, $(C_s- D_s)^2 \in \R$, hence $\Delta(\overline A)>0.$ Therefore, $\overline A$ has two distinct eigenvalues and hence is diagonalizable. For the sake of completion, the two eigenvalues are
\[
\lambda_i = \frac{C_s+D_s\pm\sqrt{(C_s-D_s)^2+4k_1k_2}}{2}.
\]
\end{proof}

\section{The general case}\label{sec:general case}

In the previous section we considered joins of $2$ circulant matrices of a special important shape. In this section, we extend our results to general finite joins of circulant matrices. In our main theorem, we completely characterize the spectrum of these matrices. First, let us introduce some notations and conventions.

Let $d, k_1, k_2, \ldots, k_d \in \N\setminus\{0\}$. Set also 
\( n=k_1+k_2+\ldots+k_d .\) 
Thus $n$ is a partition of $n$ into $d$ non-zero summands. We shall consider $n \times n$ matrices of the following form
\[
A=\left(\begin{array}{c|c|c|c}
C_1 & a_{1,2}\ind & \cdots & a_{1,d}\ind \\
\hline
a_{2,1}\ind & C_2 & \cdots & a_{2,d}\ind \\
\hline
\vdots & \vdots & \ddots & \vdots \\
\hline
a_{d,1}\ind & a_{d,2}\ind & \cdots & C_d
\end{array}\right),
\]
where for each $1 \leq i,j \leq d$ $C_i=Circ(c_{i,0},\dots,c_{i,k_{i-1}})$ is a circulant matrix of size $k_i \times k_i$, and $a_{i,j}\ind$ is a $k_i \times k_j$ matrix with all entries equal to a constant $a_{i,j}$.

We have a direct generalization of Proposition \ref{prop:block_circulant}:

\begin{prop}\label{prop:circulant eigenvectors}
For each $1 \leq i \leq d$ and  $1 \leq j \leq k_i-1$ let 
\[ \begin{aligned}
w_{i,j}&=\vec{0}_{k_1} \conc \ldots \conc\vec{0}_{k_{i-1}}\conc v_{k_i,j} \conc \vec{0}_{k_{i+1}} \conc \ldots \conc \vec{0}_{k_d}\\
&= \vec{0}_{k_1} \conc \ldots \conc \vec{0}_{k_{i-1}} \conc  \underbrace{(1, \omega_{k_i}^j, \omega_{k_i}^{2j}, \ldots, \omega_{k_i}^{(k_i-1)j})^T}_{\text{$i$-th block}} \conc\, \vec{0}_{k_{i+1}} \conc \ldots \conc \vec{0}_{k_d}.
\end{aligned}
\] 
Then $w_{i,j}$ is an eigenvector of $A$ associated with the eigenvalue 
\[ \lambda _{j}^{C_i}=c_{i,0}+c_{i, k_i-1}\omega _{k_i}^{j}+c_{i, k_i-2}\omega_{k_i} ^{2j}+\dots +c_{i,1}\omega_{k_i} ^{(k_i-1)j} \]

Furthermore, the system of $\sum_{i=1}^d k_i -d$ eigenvectors $\{w_{i,j} \}$ is linearly independent. 
\end{prop} 

We introduce the following terminology.
\begin{definition}
Keeping the previous notation, we will refer to the $w_{i,j}$'s and to the associated eigenvalues as the \emph{circulant eigenvectors and eigenvalues} of $A$.

Let $\lambda_1, \lambda_2, \ldots \lambda_d$ be the (not necessarily distinct) remaining eigenvalues of $A$.
The \textit{reduced characteristic polynomial} of $A$ is 
\[ \overline{p}_{A}(X)=\prod_{i=1}^d (X-\lambda_i)=\dfrac{p_{A}(X)}{\prod_{\substack{1 \leq i \leq d, \\ 1 \leq  j \leq k_i-1}} (X-\lambda_{j}^{C_i})} .\] 
\end{definition} 

Motivated by the findings of Section \ref{sec:2 circulant}, we look for the missing eigenvectors of $A$ in a special form, namely 
\begin{equation}\label{eq:eigenvector} v= (x_1)_{k_1} \conc \ldots \conc(x_i)_{k_i} \conc \ldots \conc (x_d)_{k_d}, \end{equation} 
where 
\[ (x_i)_{k_i}= \underbrace{(x_i, \ldots, x_i)^{T}}_{\text{$k_i$ terms}}.\] 

For $1 \leq i \leq d$, we denote the row sum of the matrix $C_i$ by
\[ C_{is}=\sum_{j=0}^{k_i-1} c_{i,j}.\]
A direct calculation shows that 
\[ A v= (C_{1s} x_1+a_{12}k_2 x_2+ \ldots+ a_{1d} k_d x_d)_{k_1} \conc \ldots \conc (a_{d1}k_1x_1+a_{d2}k_2x_2+\ldots+C_{ds} x_d)_{k_d} .\] 

Therefore, the equation $A v=\lambda v$ can be equivalently written as 
\[ \overline A (x_1, x_2, \ldots, x_d)^{T}= \lambda (x_1, \ldots, x_d)^{T} .\] 

where $\overline A$ is the $d\times d$ matrix
\[ \overline A= 
\begin{pmatrix}
C_{1s} & a_{12}k_2 & \cdots & a_{1n}k_d \\
a_{21}k_1 & C_{2s} & \cdots & a_{2n}k_d \\
\vdots  & \vdots  & \ddots & \vdots  \\
a_{d1}k_1 & a_{d2}k_2 & \cdots & C_{ds} 
\end{pmatrix} .\] 

In other words, an eigenvector of $A$ of the form \eqref{eq:eigenvector} can be ``condensed'' to an eigenvector $(x_1, \ldots, x_d)^{T}$ of $\overline A$ with respect to the same eigenvalue. A strong converse statement also holds: to prove it, we need a preliminary lemma. 
\begin{lem}\label{lem:det(M) and det(X)}
Let $X=(x_{ij})$ be a $d \times d$ matrix.  Let $M$ be the $(k_1+\ldots+k_d) \times (k_1+\ldots+k_d)$ matrix formed by the following column vectors (in this order) 

\[\begin{array}{c} (x_{11})_{k_1} \conc \ldots \conc (x_{d1})_{k_d},\; w_{j, 1}\; (1 \leq j \leq k_1-1) , \\ 
(x_{12})_{k_1} \conc \ldots \conc (x_{d2})_{k_d},\; w_{j, 2}\; (1 \leq j \leq k_2-1 ),\\ 
(x_{1n})_{k_1} \conc \ldots \conc (x_{dd})_{k_d},\; w_{j, d}\;( 1 \leq j \leq k_d-1).
\end{array}
\] 

Then 
\[ \det(M)=\det(E_{k_1}) \ldots \det(E_{k_d}) \det(X),\]
where $E_k$ is the nonsingular matrix
\[ E_k = \begin{pmatrix}1&1&1&\dots &1\\1& \omega_k&\omega_k^{2}&\dots &\omega_k^{k-1}\\1&\omega_k^2&\omega_k^{4}&\dots &\omega_k^{2(k-1)}\\\vdots &\vdots &\vdots &\ddots &\vdots \\1& \omega_k^{k-1}&\omega_k^{2(k-1)}&\dots &\omega_k^{(k-1)(k-1)}\end{pmatrix} \] 
In particular, $M$ is non-singular iff $X$ is non-singular.
\end{lem} 
\begin{proof}
By induction and the Laplace expansion formula, analogously to the proof of Proposition \ref{prop:diag}.
\end{proof}

\begin{definition}
For $v = (x_1,\dots,x_d)^T\in \C^d$, $k_1,\dots,k_d\in\N\setminus\{0\}$ and $n=k_1+\dots+k_d$, we refer to the vector 
\[
v^\otimes = (\underbrace{x_1, \ldots, x_1}_{\text{$k_1$ terms}},\dots,\underbrace{x_d, \ldots, x_d}_{\text{$k_d$ terms}})^{T}\in \C^n
\]
as the \emph{tensor expansion} of $v$.
\end{definition}
\begin{prop}
The tensor expansions of the generalized eigenspaces of $\overline A$ are generalized eigenspaces of $A$. More precisely, if $(\overline A-\lambda I)^mv=0$ for some $v\in\C^d$, $\lambda\in Spec(\overline A)$ and $m\in\N$, then $(A-\lambda I)^mv^\otimes=0$.
\end{prop}
\begin{proof}
Note preliminarily that, by the construction of the matrix $\overline A$, for any $v\in\C^d$ and any $\lambda\in\C$
\begin{equation}\label{eq:tensor expansion}
\left[(\overline A-\lambda I)v\right]^\otimes = (A-\lambda I)v^\otimes.
\end{equation}
We proceed by induction on $m$. The case $m=1$, that is, of ordinary eigenvectors, is a direct consequence of Equation \eqref{eq:tensor expansion}.
Now suppose by inductive hypothesis that for $w\in\C^d$ and $\lambda\in Spec(\overline A)$
\begin{equation}
\label{eq:generalized eigenvectors}
(\overline A-\lambda I)^{m-1}w=0 \Rightarrow (A-\lambda I)^{m-1}w^\otimes=0,
\end{equation}
and let $v\in\C^d$ satisfy $(\overline A-\lambda I)^mv=0$. Then $w=(\overline A-\lambda I)v$ satisfies the premise of \eqref{eq:generalized eigenvectors}.
Consequently,
\[
(A-\lambda I)^mv^\otimes=(A-\lambda I)^{m-1}\left((A-\lambda I)v^\otimes\right)\stackrel{\eqref{eq:tensor expansion}}{=}(A-\lambda I)^{m-1}w^\otimes\stackrel{\eqref{eq:generalized eigenvectors}}{=}0.
\]
\end{proof}
\begin{prop}
Let $\{u_1,\dots,u_d\}$ be a basis of generalized eigenvectors of $\overline A$. Then the set made of the circulant eigenvectors $w_{i,j}$ of $A$ introduced in Proposition \ref{prop:circulant eigenvectors}, together with $u_1^\otimes,\dots,u_d^\otimes$, is linearly independent.
\end{prop}
\begin{proof}
We claim that $\mathrm{span}\{w_{i,j}\mid i=1,\dots,d,j=1,\dots,k_i\}\cap\mathrm{span}\{u_1^\otimes,\dots,u_d^\otimes\}=\{0\}$. In fact, the latter span is included in the subspace $U=\{v=(y_1,\dots,y_n)\in\C^n\mid y_1=\dots=y_{k_1}, y_{k_1+1}=\dots=y_{k_1+k_2},\dots,y_{k_1+\dots+k_{d-1}+1}=\dots=y_{k_1+\dots+k_d}\}$. If by contradiction we assume a nontrivial linear combination $\sum_{i=1}^d\sum_{j=1}^{k_i}\alpha_{i,j}w_{i,j}$ to lie in $U$, then by direct inspection each partial linear combination $\sum_{j=1}^{k_i}\alpha_{i,j}w_{i,j}$ (with fixed $i$) has to lie in $U$.
Suppose, without loss of generality, that the partial linear combination $\sum_{j=1}^{k_1}\alpha_{1j}w_{1,j}$ is nontrivial. Then, for some $c\in\C$, \[\sum_{j=1}^{k_1}\alpha_{1,j}w_{1,j}=(\underbrace{c,c,\dots,c}_{k_1},0,\dots,0)^T,\] which implies a nontrivial linear relation between $v_{k_1,0}, v_{k_1,1}, v_{k_1,k_1-1}$, in contradiction with Remark \ref{rem:Fourier modes}.

\end{proof}
Now a counting argument on dimensions shows that there is no room for any (generalized) eigenvector of $A$ other than the circulant eigenvectors and the tensor expansions of the (generalized) eigenvectors of $\overline A$. We collect several direct consequences of this fact.
\begin{cor} \label{prop:diagonal_n}
$A$ is diagonalizable if and only if $\overline A$ is. In particular, if $\overline A$ is diagonalizable with eigenvalue-eigenvector pairs $(\lambda_1,v_1),\dots,(\lambda_d,v_d)$,
then $A$ is diagonalizable with the following system of eigenvalue-eigenvector pairs:
\[
\begin{array}{c}
(\lambda_j,v_j^\otimes) \text{ for }1 \leq j \leq d,\\
(\lambda^{C_1}_j, w_{1,j}) \text{ for }1 \leq j \leq k_1-1,\\
(\lambda^{C_2}_j, w_{2,j}) \text{ for }1 \leq j \leq k_2-1,\\
\dots\\
(\lambda^{C_d}_j, w_{d,j}) \text{ for }1 \leq j \leq k_d-1.\\
\end{array}  \]
\end{cor}

\begin{cor}
The reduced characteristic polynomial of $A$ coincides with the characteristic polynomial of $\overline A$, namely
\[ \overline{p}_{A}(X)=p_{\overline A}(X) .\] 
\end{cor}
 
\section{Some applications to network theory}
In this section, we apply the main results to study the spectrum of several (directed) graphs by the join and edge-removal procedures. In particular, we provide a conceptual explanation for the spectrum of the graph described in the second section.

First, we recall a graph construction, namely the join construction (see \cite[Chapter 2]{[Harary]} and \cite{[Zykov]}).
\begin{definition}
Let $G=(V(G),E(G)), H=(V(H),E(H))$ be two graphs. The join of $G$ and $H$, denoted by $G + H$, is the graph with vertex set $V = V (G) \cup V (H)$, and in which two vertices $u$ and $v$ are adjacent if and only if
\begin{itemize}
\item $u, v \in V(G)$ and $uv \in E(G)$.
\item $u,v \in V(H)$ and $uv \in E(H)$.
\item $u \in V(G)$ and $v \in V(H)$.
\item $u \in V(H)$ and $v \in V(G)$.
\end{itemize}
\end{definition} 
Here is a pictorial illustration: 
\begin{center}
\includegraphics[trim={0 15cm 0 5cm},clip,scale=0.1]{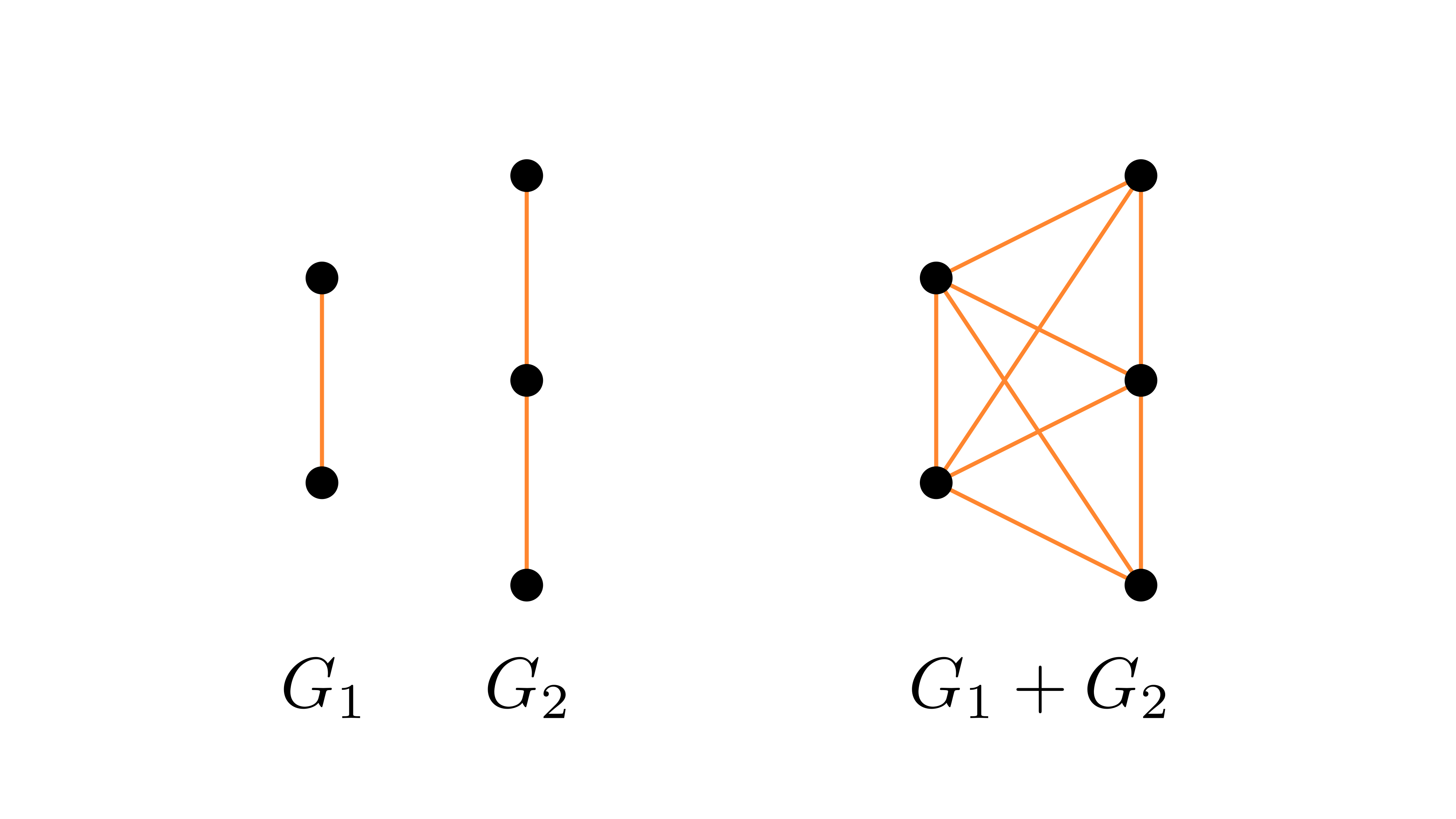}
\end{center}

Let $A(G)$ and $A(H)$ be the adjacency matrices of $G$ and $H$ respectively. Then the adjacency matrix of $G + H$ is given by 
\[ A(G + H)=\begin{pmatrix}
   A(G) & \ind_{k_1, k_2} \\
   \ind_{k_2, k_1}  &  A(H)
\end{pmatrix}, \] 
with $k_1=|V(G)|$ and $k_2=|V(H)|$. Therefore, in any case in which $A(G)$ and $A(H)$ are circulant, the spectrum of $A(G+H)$ is completely determined by Propositions \ref{prop:block_circulant}, \ref{prop:converse}, \ref{prop:explicit eigenvalues}. Here are some interesting instances.
\begin{expl}[Ring graphs]
For two positive integers $k,m$, the \textit{ring graph} $\RG(k,m)$ is the undirected graph whose $k$ vertices can be arranged in a circle in such a way that each vertex is connected to its $m$ closest neighbours on each side (with the understanding that, for $k\leq2m+1$, $\RG(k,m)$ is the complete graph $K_k$). In particular, ring graphs are regular. We choose a total order of the vertices which goes along the aforementioned circle. This produces a circulant adjacency matrix, whose eigenvalue corresponding to the eigenvector $(1,1,\dots,1)^T$ is the graph valency $2m$.

Consequently, if $k_1>2m_1+1$ or $k_2>2m_2+1$, the spectrum of $\RG(k_1,m_1) + \RG(k_2,m_2)$ is the union of three multisets 
\[(\Spec(\RG(k_1,m_1) \setminus \{2m_1 \}) \cup  (\Spec(\RG(k_2,m_2) \setminus \{2m_2 \}) \cup \{\lambda_1, \lambda_2 \}  ,\]
with
\begin{align*}
    \lambda_1, \lambda_2 &=\frac{(2m_1+2m_2) \pm \sqrt{(2m_1-2m_2)^2+4k_1k_2}}{2}\\
    &=m_1+m_2 \pm \sqrt{(m_1-m_2)^2+k_1k_2}.
\end{align*}

For the sake of completion, if $k_1\geq2m_1+1$ and $k_2\geq2m_2+1$, then clearly $\RG(k_1,m_1) + \RG(k_2,m_2)$ is the complete graph on $k_1+k_2$ vertices, so its spectrum is well known.
\end{expl}

\begin{expl}[Cycle removal 1]
Let us consider the graph obtained by removing an undirected cycle of length $k$ from the complete graph $K_n$ with $n>k$. Up to a reordering of the vertices, the resulting graph is the join of a circulant graph $G$, with $k$ vertices and $k-3$ edges and with adjacency matrix $A(G)=Circ(0,0,1,1\dots,1,0)$,
and the complete graph $H=K_{n-k}$.
Since
\[
\Spec(G) = \left\{\left[\sum_{r=2}^{k-2}\omega_k^{rj}\right]_1\;\middle|\; j = 0,\dots,k\right\}
\]
and
\[ \Spec(H)=\{ [-1]_{n-k-2}, [n-k-1]_{1} \}, \]
with lower indices after square brackets denoting algebraic multiplicity, the spectrum of $G+H$ is the multiset
\[
\left\{\left[\sum_{r=2}^{k-2}\omega_k^{rj}\right]_1\;\middle|\; j = 1,\dots,k\right\}\cup\{[-1]_{n-k-2}\}\cup\{\lambda_1,\lambda_2\}
\]
with
\[ \lambda_1, \lambda_2= \frac{(n-4)+\sqrt{(n+2)^2-8k}}{2}.\]

\end{expl}

\begin{expl}[Cycle removal 2]
Similarly, the graph obtained by removing a directed cycle of length $k$ from the complete graph $K_n$ with $n>k$ is the join of a circulant graph $G$, with $k$ vertices and $k-2$ edges and with adjacency matrix $A(G)=Circ(0,1,1\dots,1,0)$, and the complete graph $H=K_{n-k}$.
Its spectrum is the multiset
\[
\left\{\left[\sum_{r=2}^{k-1}\omega_k^{rj}\right]_1\mid j = 1,\dots,k\right\}\cup\{[-1]_{n-k-2}\}\cup\{\lambda_1,\lambda_2\}
\]
with
\[ \lambda_1, \lambda_2= \frac{(n-3)+\sqrt{(n+1)^2-4k}}{2}.\]

\end{expl}

\section{Applications to non-linear dynamics on oscillator networks}
To illustrate potential applications of this approach, we can now consider a dynamical system on the join of several circulant graphs. Specifically, we consider oscillators coupled on a graph on a matrix $A$, defined by joining $d$ identical circulant graphs. We consider the Kuramoto model:
\begin{equation}
    \frac{d\theta_{i}}{dt} = \omega_{i} + \epsilon \sum\limits_{j=1}^{N} A_{ij} \sin{(\theta_{j} - \theta_{i})},
    \label{eq:main_kuramoto}
\end{equation}
which is a central tool in the description of synchronization in nature, from the behavior of insects (see \cite{[buck1998]}, \cite{[ermentrout]}), patterns of social behavior (see \cite{[pulchino1]}, \cite{[pulchino2]}), neural systems (see \cite{[breakspear1]}, \cite{[cabral]}), and physical systems (see \cite{[val]}, \cite{[wie]}). Here, $\theta_i$ is the state of oscillator $i \in [1,N]$ at time $t$, $\omega_i$ is the intrinsic angular frequency, $\kappa$ scales the coupling strength, and element $a_{ij}$ represents the weighted connection between oscillators $i$ and $j$. We focus on the case where all oscillators have the same natural frequency, that is, $\omega_i= \omega~\text{for all}~i \in [1,N]$. Under this condition, we can assume further that $\omega=0.$

An important question in this area is the study of  equilibrium points on a network of Kuramoto oscillators (see \cite{[Taylor12]}, \cite{[Townsend]}). Furthermore, it is known that the stability of these equilibrium points depends strongly on the specific pattern of connections, highlighting the importance of the network's structure on the Kuramoto dynamics (see \cite{[Townsend]}). In \cite{[KO4]}, we utilize a algebraic approach to study equilibrium points of this dynamical system. By studying a related complex-valued model introduced in \cite{[KO1]}, we prove the following theorem. 

\begin{thm} (See \cite[Proposition 2] {[KO4]}) \label{prop:twisted_state_KM}
Suppose $\bm{x}_0=e^{\i \bm{\theta}_0}$ is an eigenvector of $\bm{A}=(a_{ij})$ associated with a real eigenvalue $\lambda$. Then $\bm{\theta}_0=(\theta_1, \theta_2, \ldots, \theta_N)$ is an equilibrium point of the following Kuramoto model.
\[ \frac{d\theta_i}{dt} = \epsilon \sum_{j=1}^{N} a_{ij} \sin( \theta_j - \theta_i ). \] 
\end{thm} 
We will now use this result and the main theorem of this article to construct networks with interesting equilibrium points. More precisely, let $C$ be a real symmetric circulant matrix of size $k \times k$. Let $A$ be a join of $d$-identical copies of $C$, namely $A$ is a network with the following weighted adjacency matrix 

\[
A=\left(\begin{array}{c|c|c|c}
C & a_{1,2}\ind & \cdots & a_{1,d}\ind \\
\hline
a_{2,1}\ind & C & \cdots & a_{2,d}\ind \\
\hline
\vdots & \vdots & \ddots & \vdots \\
\hline
a_{d,1}\ind & a_{d,2}\ind & \cdots & C
\end{array}\right),
\]
Let $(\varphi_1, \varphi_2, \ldots, \varphi_d) \in [-\pi, \pi]^d$. For each $1 \leq j \leq k-1$ and $1 \leq i \leq d$ let us define
\[\omega_{i,j}^{(\varphi_i)}= e^{\i \varphi_i} \omega_{i,j}, \]
where $\omega_{i,j}$ is an eigenvector of $A$ associated with the eigenvalue $\lambda_{j}^C$ as described in Proposition \ref{prop:circulant eigenvectors}. Because these eigenvectors for a fixed $j$ are associated with a single eigenvalue, their sum $\sum_{i=1}^d \omega_{i,j}^{(\varphi_i)}$
is also an eigenvector associated with the eigenvalue $\lambda^{C}_j$. Note further that by the definition of $\omega_{i,j}$, we have 
\[ \sum_{i=1}^d \omega_{i,j}^{(\varphi_i)}=e^{\i \bm{\theta}_{0,j}^{(\varphi_1, \ldots, \varphi_d)}}, \] 
where 
\[ \bm{\theta}_{0,j}^{(\varphi_1, \ldots, \varphi_d)}=\left(\varphi_1, \frac{2 \pi j}{k}+\varphi_1, \ldots, \frac{2 \pi(k-1)j}{k}+\varphi_1, \ldots, \varphi_d, \frac{2 \pi j}{k}+\varphi_d, \ldots, \frac{2 \pi(k-1)j}{k}+\varphi_d \right)^T. \] 

By Theorem \ref{prop:twisted_state_KM}, we conclude that 
\begin{prop}
For all $1 \leq j \leq k-1$ and $(\varphi_1, \varphi_2, \ldots, \varphi_d) \in [-\pi, \pi]^d$, $\bm{\theta}_{0,j}^{(\varphi_1, \ldots, \varphi_d)}$ is an equilibrium point of the KM associated with the adjacency matrix $A$.

\end{prop}

\section*{Acknowledgments}
\noindent This work was supported by BrainsCAN at Western University through the Canada First Research Excellence Fund (CFREF), the NSF through a NeuroNex award (\#2015276), the Natural Sciences and Engineering Research Council of Canada (NSERC) grant R0370A01, SPIRITS 2020 of Kyoto University, Compute Ontario (computeontario.ca), and Compute Canada (computecanada.ca). J.M.~gratefully acknowledges the Western University Faculty of Science Distinguished Professorship in 2020-2021.

\end{document}